\documentclass[11pt,reqno]{amsart}
\usepackage{amsfonts,amstext}
\usepackage{amsmath,amstext,amssymb,amscd,mathrsfs,amsthm}
\usepackage{enumerate}
\usepackage[dvipsnames,usenames]{xcolor}

\newtheorem{thm}{Theorem}[section]
\newtheorem{lem}[thm]{Lemma}
\newtheorem{cor}[thm]{Corollary}

\newtheorem{assumption}[thm]{Assumption}
\theoremstyle{definition}
\newtheorem{defn}[thm]{Definition}

\theoremstyle{remark}
\newtheorem{rmk}[thm]{Remark}
\numberwithin{equation}{section}

\usepackage[allbordercolors={1 1 1}]{hyperref}
\usepackage[disable]{todonotes}

\def\grad{\nabla}

\renewcommand{\H}[1][1]{W^{#1,2}(\Omega)}

\renewcommand{\to}{\mathrel{\rightarrow}}

\newcommand{\R}{\mathbb{R}}

\def\g{\gamma}

\def\t{\tau}

\def\G{\Gamma}
\def\O{\Omega}

\newcommand{\norm}[1]{\left\|#1\right\|}
\newcommand{\abs}[1]{\left|#1\right|}

\def\<{\langle}
\def\>{\rangle}

\def\H{H^1_{\G_0}}
\begin{document}

\author[B. Feng]{Baowei Feng}
\address{Department of Mathematics,
Southwestern University of Finance and Economics,
Chengdu 611130, P. R. China }
 \email{bwfeng@swufe.edu.cn}

\author[Y. Guo]{Yanqiu Guo}
\address{Department of Mathematics and Statistics, Florida International University, Miami FL 33199, USA }
 \email{yanguo@fiu.edu}

\author[M. A. Rammaha]{Mohammad A. Rammaha}
\address{Department of Mathematics, University of Nebraska--Lincoln, Lincoln, NE  68588-0130, USA} \email{mrammaha1@unl.edu}

\title[Blow-up of a structural acoustics model]{Blow-up of a structural acoustics model}

\date{January 1, 2023}

\maketitle

\begin{abstract} This article studies the finite time blow-up of weak solutions to a structural acoustics model consisting of a semilinear wave equation defined on a bounded domain $\Omega\subset\mathbb{R}^3$ which is strongly coupled with a  Berger plate equation acting on 
the elastic wall, namely, a flat portion of the boundary. The system is influenced by several competing forces, including boundary and interior source and damping terms. We stress that the power-type source term acting on  the wave equation is  allowed to  have a \emph{supercritical} exponent, in the sense that its associated Nemytskii operators is not locally Lipschitz from $H^1$ into $L^2$. 
In this paper, we prove the blow-up results for weak solutions when the source terms are stronger than damping terms, by considering two scenarios of the initial data: (i) the initial total energy is negative; (ii) the initial total energy is positive but small, while the initial quadratic energy is sufficiently large. The most significant challenge in this work arises from the coupling of the wave  and plate equations on the elastic wall.
\end{abstract}

\maketitle

\section{Introduction}\label{S1}

We study the finite time blow-up for a structural acoustics model influenced with nonlinear forces. Precisely, we consider the following coupled system of nonlinear PDEs:
\begin{align}\label{PDE}
\begin{cases}
u_{tt}-\Delta u+u +g_1(u_t)=f(u) &\text{ in } \O \times (0,T),\\[1mm]
w_{tt}+\Delta^2w+g_2(w_t)+u_t|_{\G}=h(w)&\text{ in }\G\times(0,T),\\[1mm]
u=0&\text{ on }\G_0\times(0,T),\\[1mm]
\partial_\nu u=w_t&\text{ on }\G\times(0,T),\\[1mm]
w=\partial_{\nu_\G}w=0&\text{ on }\partial\G\times(0,T),\\[1mm]
(u(0),u_t(0))=(u_0,u_1),\hspace{5mm}(w(0),w_t(0))=(w_0,w_1),
\end{cases}
\end{align}
where the  initial data reside in the finite energy space, i.e., 
$$(u_0, u_1)\in H^1_{\G_0}(\O) \times L^2(\O) \, \text{ and }(w_0, w_1)\in H^2_0(\G)\times L^2(\G).$$
The space $H^1_{\G_0}(\O)$ defined in (\ref{def-H}) consists of all $H^1$ functions that vanish on $\Gamma_0$.

Here,  $\O\subset\R^3$ is a bounded, open, connected domain with smooth boundary 
$\partial\O=\overline{\G_0\cup\G}$, where $\G_0$ and $\G$ are two disjoint, open, connected sets of positive Lebesgue  measure.  Moreover, $\G$ is a \emph{flat} portion of the boundary of $\O$ and is referred to as the elastic wall. The part $\G_0$ of the boundary $\partial\O$ describes a rigid  wall, while  the coupling takes place on the flexible wall  $\G$. 

The nonlinearities $f$ and $h$ are source terms acting on the wave and plate equations respectively, where the source term $f(u)$ is of a supercritical order, in the sense that its associated Nemytskii operator is not locally Lipschitz from $H^1_{\G_0}(\O) $ into $L^2(\Omega)$. In the case of the 3D domain $\Omega$, the supercritical order means that the exponent of the power-like function $f$ is larger than 3. We stress that both source terms $f(u)$ and $h(w)$ are allowed to have ``bad" signs  which  may cause instability (blow up) in  finite time.  In addition, the system is influenced by two other forces, namely $g_1(u_t)$ and $g_2(w_t)$ representing frictional damping terms acting on the wave and plate equations, respectively. The vectors $\nu$ and $\nu_\G$ denote the outer normals to $\G$ and $\partial\G$, respectively.

Models such as (\ref{PDE}) arise in the context of modeling gas pressure in an 
acoustic chamber which is surrounded by a combination of rigid and 
flexible walls. The pressure in the chamber is described by the solution 
to a wave equation, while vibrations of the flexible wall are described 
by the solution to a coupled a Berger plate equation.

PDE models describing structural acoustic interaction have rich history. These models are well known in both the physical and mathematical literature and go back to the canonical models considered in \cite{Beale76,Howe1998}. In the context of stabilization and controllability of structural acoustics models there is a very large body of literature. We refer the reader to the monograph by Lasiecka \cite{Las2002} which provides a  comprehensive overview and quotes many works on these topics. Other related contributions include \cite{Avalos2,Avalos1,Avalos3,Avalos4,Cagnol1,MG1,MG3,LAS1999}. However, the finite time blow-up for structural acoustics models under the influence of nonlinear forces has not been studied in the literature, and so we address this issue in this paper.

Our goal is to understand the source-damping interactions in the structural acoustics model (\ref{PDE}), 
and how these interactions affect the behaviors of weak solutions. 
The local well-posedness of weak solutions to system (\ref{PDE}) was proved by 
Becklin and Rammaha in \cite{Becklin-Rammaha2}, in which they also showed the global existence if the damping are more dominant than the source terms. In our paper \cite{Feng-Guo-Rammaha1}, by using the potential well theory, we proved the global existence of weak solutions and estimated the energy decay rates, provided the initial data come from the stable part of the potential well. 
In the present manuscript, we shall demonstrate the blow-up phenomena of weak solutions when the source terms are stronger than damping terms, and we consider two cases of the initial data: (i) the initial total energy is negative, which means that the initial potential energy due to the nonlinear forces is sufficiently large; (ii) the initial total energy is positive but small enough, while the initial quadratic energy is large. In this case, the initial data come from the unstable part of the potential well. To understand these blow-up phenomena intuitively, one can imagine that a nonlinear force continuously acts on the elastic wall to increase its vibration and simultaneously another nonlinear force acts on the gas inside the acoustic chamber to increase its pressure, and these forces surpass the damping effects, then the system collapses at some finite time.

The difficulty for proving the blow-up of weak solutions to system (\ref{PDE}) comes from the coupling of the wave equation and the plate equation on the elastic wall, i.e., the flat portion of the boundary. We notice that the coupling of these two evolution equations in (\ref{PDE}) are through the term $u_t|_{\Gamma}$ where 
$\Gamma$ is the elastic wall. Since we consider weak solutions of the wave equation, $u_t$ belongs to $L^2(\Omega)$, while a generic $L^2$ function may not have a well-defined trace on the boundary of $\Omega$. In system (\ref{PDE}), the term  $u_t|_{\Gamma}$ is defined in a weak sense via the plate equation. But we do not have an appropriate estimate for the $L^2(\Gamma)$ norm of $u_t|_{\Gamma}$. 
Therefore, throughout the proof of our blow-up results, we strive to prevent directly estimating $u_t|_{\Gamma}$, and the idea is to convert $u_t|_{\Gamma}$  to a different term by taking advantage of the structure of the equation. Note, the basic strategy for proving the blow-up is to show the function $Y(t)=G^{1-a}(t)+\varepsilon N'(t)$ approaches infinity in finite time. The function $Y(t)$ was used in \cite{GT}, where $G(t)$ was the negative of the total energy. But, in \cite{GT} and many other related works in the literature, $N(t)$ is usually defined as the $L^2$ norm of the unknown function. However, in our argument, we use a trick by including an additional term $\int^t_0\int_\Gamma\gamma u(\tau)\cdot
w(\tau)d\Gamma d\tau$ in $N(t)$ (see (\ref{def-N})).  This extra term helps us to convert the troublesome term $u_t|_{\Gamma}$ in our estimate to a well-behaved term $w_t$, where the $L^2(\Gamma)$ norm of $w_t$ is part of the energy.

We must point out that in the original linear structural acoustics model, $u$ satisfies the wave equation $u_{tt}-\Delta u=0$. However, in order to resolve some technical difficulty occurred during our proof for the blow-up results, we add a term $u$ to the linear part, and the linear operator in our system (\ref{PDE}) becomes $u_{tt}-\Delta u+u$, which usually appears in a Klein-Gordon equation.
The extra term $u$ in the linear operator is useful when we estimate the $L^2(\Gamma)$ norm of the trace of $u$ in (\ref{6-14-2}) since it allows us to obtain precise coefficients on the right-hand side of the inequality, which is critical for our argument.

Source-damping interactions have important applications. On one hand, in control theory, one may use damping terms to stabilize the system. On the other hand, one can create instability by strengthening the source terms. The interesting source-damping interactions in wave equations have been illustrated by a pioneering work by Georgiev and Todorova \cite{GT}. 
Bociu and Lasiecka wrote a series of papers \cite{BL3, BL2, BL1} to study wave equations with \emph{supercritical} source and damping terms acting in the interior and on the boundary of the domain. Also, Guo \cite{Guo} proved the global well-posedness of a 3D wave equation with a source term of an arbitrarily large exponent, as long as the frictional damping term is strong enough to suppress the growth of solutions due to the source term. One may also refer to papers \cite{GR1, GR2, GR} for source-damping interactions in coupled wave equations.

The content of the paper is organized as follows. In Section \ref{sec-2}, we state well-posedness results from \cite{Becklin-Rammaha2} by Becklin and Rammaha, and our previous results on global existence and energy decay of weak solutions in \cite{Feng-Guo-Rammaha1}. Moreover, we state  main results of this manuscript, namely, finite-time blow-up of weak solutions. In Section \ref{blowup1}, we prove the blow-up of weak solutions by assuming the initial total energy is negative. In Section \ref{blowup2}, we show the finite-time blow-up by supposing the initial total energy is positive.

\vspace{0.1 in}

\section{Preliminaries and main results} \label{sec-2}

\subsection{Notation}

\noindent{}Throughout the paper the following notational conventions for $L^p$ space norms and standard inner products will be used:
\begin{align*}
&||u||_p=||u||_{L^p(\O)}, &&(u,v)_\O = (u,v)_{L^2(\O)},\\
&|u|_p=||u||_{L^p(\G)},&&(u,v)_\G = (u,v)_{L^2(\G)}.
\end{align*}
We also use the notation  $\g u$ to denote the \emph{trace} of $u$ on $\G$. As is customary, $C$  always denotes a generic positive constant which may change from line to line.

\noindent{}Further, we put 
\begin{align} \label{def-H}
\H(\O):=\{u\in H^1(\O):u|_{\G_0}=0\},
\end{align}
and
$\norm{u}_{\H(\O)}:= (\norm{\grad u}^2_2+\|u\|^2_2)^{\frac{1}{2}}$.
It is well-known that the standard norm
$\norm{u}_{\H(\O)}$ is equivalent to $\norm{\grad u}_2$. 
For a similar reason, we put $\norm{w}_{H_0^2(\G)}= \abs{\Delta w}_2$. 

In the proof, the following Sobolev imbeddings will be used:
$H_{\Gamma_0}^1(\O)\hookrightarrow L^{6}(\O)$ and $H^1(\Gamma)\hookrightarrow L^q(\G)$ for any $1\leq q<\infty$.

\subsection{Well-posedness of weak solutions} Throughout the paper, we study  (\ref{PDE}) under the following assumptions.

\begin{assumption}\label{ass} 
\hfill
\begin{description}
\item[Damping] $g_1, \, g_2:\R \rightarrow \R$ are continuous and monotone increasing functions with $g_1(0)=g_2(0)=0$.  In addition, the following growth conditions at infinity hold: there exist positive constants $\alpha$ and $\beta$ such that, for $|s|\geq1$,
\begin{align*}
&\alpha|s|^{m+1}\leq g_1(s)s\leq\beta|s|^{m+1},\text{ with }m\geq1,\\
&\alpha|s|^{r+1}\leq g_2(s)s\leq\beta|s|^{r+1},\text{ with }r\geq1.
\end{align*}
\item[Source terms] $f$ and $h$ are functions in $C^1(\R)$ such that
\begin{align*}&|f'(s)|\leq C(|s|^{p-1}+1),\text{ with }1\leq p< 6,\\
&|h'(s)|\leq C(|s|^{q-1}+1),\text{ with }1\leq q<\infty.
\end{align*}
\item[Parameters] \,\, $p\frac{m+1}{m}<6$.
\end{description}
\end{assumption}

The following assumption will be needed for establishing an uniqueness result.
\begin{assumption}\label{ass-2}
For $p>3$, we assume that  $f\in C^2(\R)$ with $|f''(u)|\leq C(|u|^{p-2}+1)$ for all $u\in\R$.
\end{assumption}

\noindent{} We begin by introducing the definition of a suitable weak solution for \eqref{PDE}.
\begin{defn}\label{def:weaksln}
A pair of functions $(u,w)$ is said to be a weak solution of \eqref{PDE} on the interval $[0,T]$ provided:
	\begin{enumerate}[(i)]
		\setlength{\itemsep}{5pt}
		\item\label{def-u} $u\in C([0,T];H^1_{\G_0}(\O))$, $u_t\in C([0,T];L^2(\O))\cap L^{m+1}(\O\times(0,T))$,
		\item\label{def-v} $w\in C([0,T];H^2_0(\G))$, $w_t\in C([0,T];L^2(\G))\cap L^{r+1}(\G\times(0,T))$,
		\item\label{def-uic} $(u(0),u_t(0))=(u_0,u_1) \in H^1_{\G_0}(\O)\times L^2(\O)$,
		\item\label{def-wic} $(w(0),w_t(0))=(w_0,w_1) \in H^2_0(\G)\times L^2(\G)$,
		\item\label{def-ws} The functions $u$ and $w$ satisfy the following variational  identities 
		for all $t\in[0,T]$: 
\begin{align}\label{wkslnwave}
(u_{t}(t),  \phi(t))_\O & - (u_1,\phi(0))_\O-\int_0^t ( u_t(\tau), \phi_t(\tau) )_\O d\tau
+\int_0^t (\nabla u(\tau), \nabla\phi(\tau) )_\O d\tau \notag \\
&+\int^t_0(u(\tau),\phi(\tau))_\Omega d\tau-\int_0^t  (w_t(\tau), \g\phi(\tau) )_\G d\tau+\int_0^t\int_\Omega g_1(u_t(\tau))\phi(\tau) dxd\tau  \notag\\
&=\int_0^t\int_\Omega f(u(\tau))\phi(\tau) dxd\tau,
\end{align}
\begin{align}\label{wkslnplt}
(w_t(t) & + \g u(t),\psi(t) )_\G  -(w_1 +\g u_0 ,\psi(0))_\G -\int_0^t (w_t(\tau), \psi_t(\tau) )_\G d\tau \notag \\
& -\int_0^t (\g u(\tau), \psi_t(\tau) )_\G d\tau+\int_0^t (\Delta w(\tau), \Delta\psi(\tau) )_\G d\tau \notag \\
&+ \int_0^t\int_{\G}g_2(w_t(\tau))\psi(\tau) d\G d\tau=\int_0^t\int_{\G}h(w(\tau))\psi(\tau) d\G d\tau,
\end{align}
for all test functions $\phi$ and $\psi$ satisfying:
$\phi\in C([0,T];H^1_{\G_0}(\O))  \cap L^{m+1}(\O\times(0,T))$, $\psi\in C\left([0,T];H^2_0(\G)\right)$  with 
$\phi_t\in L^1(0,T;L^2(\O))$, and $\psi_t\in L^{1}(0,T;L^2(\G))$.
	\end{enumerate}
\end{defn}

Our work in this paper is based on the existence results which were established in \cite{Becklin-Rammaha2} by Becklin and Rammaha. For the reader's convenience, we first summarize the important results in \cite{Becklin-Rammaha2}. 

\begin{thm} [{\bf Local and  global weak solutions \cite{Becklin-Rammaha2}}]  \label{t:1}
	Under the validity of Assumption \ref{ass}, then there exists a local weak soluition $(u,w)$ to \eqref{PDE} defined on $[0,T_0]$ for some $T_0>0$ depending on the initial energy $E(0)$, where
the quadratic energy $E(t)$ is given by
\begin{align} \label{def-qua}
E(t):=\frac{1}{2}\left(\|u_t(t)\|_2^2+\|\nabla u(t)\|_2^2 +\|u(t)\|^2_2 +|w_t(t)|_2^2+|\Delta w(t)|_2^2\right).
\end{align}
\begin{itemize}
\item ~~$(u,w)$ satisfies the following energy identity for all $t\in [0,T_0]$:
\begin{align}\label{energy}
	E(t)&+\int_0^t\int_\O g_1(u_t)u_tdxd\t+\int_0^t\int_\G g_2(w_t)w_td\G d\t\notag\\
	&=E(0)+\int_0^t \int_\O f(u)u_tdxd\t+\int_0^t\int_\G h(w)w_td\G d\t.
\end{align}
\item In addition to Assumption  \ref{ass}, if we assume that $u_0\in L^{p+1}(\O)$, $p\leq m$ and $q\leq r$, then the said solution $(u,w)$  is a global weak solution and $T_0$ can be taken arbitrarily large.

\item If Assumptions \ref{ass} and \ref{ass-2} are valid, and if we further assume that $u_0\in L^{\frac{3(p-1)}{2}}(\O)$, then weak solutions of \eqref{PDE} are unique.

\item If Assumption \ref{ass} is valid, and if we additionally assume that $u_0\in L^{3(p-1)}(\O)$ and $m \geq 3p-4$ when  $p>3$, then weak solutions of \eqref{PDE} are unique.

\end{itemize}
\end{thm}

\vspace{0.1 in}

\subsection{Potential well solutions}
In this subsection we briefly discuss the potential well theory which originates from the theory of elliptic equations. In order to do so, we need to impose additional assumptions  on the source terms $f(u)$ and  $h(w)$.

\begin{assumption}\label{ass1} 
\hfill
\begin{itemize}
\item
 There exists a nonnegative function
$F(u)\in C^1(\mathbb{R})$ such that $F'(u)=f(u)$, and $F$ is
homogeneous of order $p+1$, i.e., $F(\lambda u)=\lambda^{p+1}F(u)$,
for $\lambda>0,\ u\in\mathbb{R}$. 
\item There exists a nonnegative function $H(s)\in C^1(\mathbb{R})$
such that $H'(s)=h(s)$, and $H$ is homogeneous of order $q+1$, i.e.,
$H(\lambda s)=\lambda^{q+1}H(s)$, for $\lambda>0,\ s\in\mathbb{R}$.
\end{itemize}
\end{assumption}

\begin{rmk}\label{rmk3-1}
From Euler homogeneous function theorem we infer that
\begin{align}\label{3-3}
uf(u)=(p+1)F(u),\ \ wh(w)=(q+1)H(w).
\end{align}
Because of Assumption \ref{ass} and the homogeneity of $F$ and $H$, we obtain that there exists a positive constant $M$ such that
\begin{align}\label{3-4}
F(u)\leq M|u|^{p+1},\ \ H(w)\leq M|w|^{q+1}.
\end{align}
Moreover, due to (\ref{3-3}), $f$ is homogeneous of order $p$ and
$h$ is homogeneous of order $q$ satisfying
\begin{align}\label{3-6}
|f(u)|\leq M(p+1) |u|^{p},\ \ |h(w)|\leq M (q+1)|w|^q.
\end{align}
\end{rmk}

Recall the quadratic energy $E(t)$ has been introduced in (\ref{def-qua}).
Now, we define the total energy $\mathcal{E}(t)$ of system \eqref{PDE} by
\begin{align}\label{3-7}
\mathcal{E}(t)
:&=E(t)-\int_\Omega F(u(t))dx-\int_\Gamma
H(w(t))d\Gamma \notag\\
&= \frac{1}{2}\left(\|u_t(t)\|_2^2+\|\nabla u(t)\|_2^2 +\|u(t)\|^2_2 +|w_t(t)|_2^2+|\Delta w(t)|_2^2\right) \notag\\
& \;\;\;\;\;- \int_\Omega F(u(t))dx-\int_\Gamma
H(w(t))d\Gamma.
\end{align}

Then, the energy identity (\ref{energy}) is equivalent to 
\begin{align} \label{energy-2}
\mathcal E(t) + \int_0^t\int_\O g_1(u_t)u_tdxd\t+\int_0^t\int_\G g_2(w_t)w_td\G d\t
=\mathcal E(0).
\end{align}

With $X:=H^1_{\Gamma_0}(\Omega)\times H^2_0(\Gamma)$, we define the
functional $\mathcal{J}:X\to \mathbb{R}$ by
\begin{align}\label{3-8}
\mathcal{J}(u,w):=\frac{1}{2}(\|\nabla u(t)\|^2_2   +\|u\|^2_2    +|\Delta
w(t)|^2_2)-\int_\Omega F(u(t))dx-\int_\Gamma H(w(t))d\Gamma,
\end{align}
where $\mathcal J(u,w)$ is the potential energy of the system. Then we
have
\begin{align}\label{3-9}
\mathcal{E}(t)=\mathcal{J}(u,w)+\frac{1}{2}(\|u_t(t)\|^2_2+|w_t(t)|^2_2).
\end{align}

The Fr\'{e}chet derivative of $\mathcal{J}$ at $(u,w)\in X$ is given
by
\begin{align}\label{fre}
\langle\mathcal{J}'(u,w),(\phi,\psi)\rangle=&\int_\Omega\nabla
u\cdot\nabla\phi dx+\int_\Gamma\Delta w\cdot\Delta\psi d\Gamma+\int_\Omega u\phi dx\nonumber\\
&-\int_\Omega f(u)\phi dx-\int_\Gamma h(w)\psi d\Gamma,
\end{align}
for $(\phi,\psi)\in X$.  The \emph{Nehari manifold} $\mathcal N$ can be defined by
\begin{align}
\mathcal{N}:=\left\{(u,w)\in X\backslash\{(0,0)\}:
\langle\mathcal{J}'(u,w),(u,w)\rangle=0\right\},\nonumber
\end{align}
which along with \eqref{fre} gives
\begin{align}\label{3-10}
\mathcal{N}=\bigg\{(u,w)\in X\backslash\{(0,0)\}: \|\nabla
u\|^2_2    +\|u\|^2_2      +|\Delta w|^2_2&=(p+1)\int_\Omega F(u)dx\nonumber\\
&\quad+(q+1)\int_\Gamma
H(w)d\Gamma\bigg\}.
\end{align}

By Lemma 2.8 in our paper \cite{Feng-Guo-Rammaha1} and Lemma 2.7 in \cite{GR2}, the depth of the potential well $d$ is positive and satisfies
\begin{align}\label{3-11}
d:=\inf_{(u,w)\in\mathcal{N}}\mathcal{J}(u,w)=\inf_{(u,w)\in
X\backslash\{(0,0)\}}\sup_{\lambda\geq0}\mathcal{J}(\lambda (u,w))>0,
\end{align}
for $1<p\leq 5,\ q>1$.

We define
\begin{align*}
\mathcal{W}&:=\{(u,w)\in X: \mathcal{J}(u,w)<d\},\\
\mathcal{W}_1&:=\left\{(u,w)\in\mathcal{W}:\|\nabla u\|^2_2     +\|u\|^2_2      +|\Delta w|^2_2>(p+1)\int_\Omega F(u)dx+(q+1)\int_\Gamma H(w)d\Gamma\right\}\nonumber\\
&\qquad\qquad\cup\{(0,0)\},\\
\mathcal{W}_2&:=\left\{(u,w)\in\mathcal{W}:\|\nabla u\|^2_2   +\|u\|^2_2   +|\Delta
w|^2_2<(p+1)\int_\Omega F(u)dx+(q+1)\int_\Gamma H(w)d\Gamma\right\}.
\end{align*}
It is obvious that $\mathcal{W}_1\cup\mathcal{W}_2=\mathcal{W}$ and
$\mathcal{W}_1\cap\mathcal{W}_2=\emptyset.$ We call $\mathcal{W}$
the potential well and $d$ is the depth of the well.
We call $\mathcal W_1$ the stable part of the potential well, and $\mathcal W_2$ the unstable part of the potential well.

For initial data coming from the stable part of the potential well, we have proved the following result of global solutions in \cite{Feng-Guo-Rammaha1}.

\begin{thm}[{\bf Potential well solutions \cite{Feng-Guo-Rammaha1}}]\label{thm3-1}
Assume that Assumption \ref{ass} and Assumption \ref{ass1}
hold. Let $1<p\leq 5$ and $q>1$. Assume further
$(u_0,w_0)\in\mathcal{W}_1$ and $\mathcal{E}(0)<d$. Then system
\eqref{PDE} admits a global solution $(u,w)$. In addition,  for
any $t\geq 0$, we have
\begin{align}
\begin{cases}
    (i)\  \mathcal{J}(u,w)\leq
\mathcal{E}(t)\leq\mathcal{E}(0),  \\
    (ii)\
(u,w)\in\mathcal{W}_1,\\
    (iii)\  E(t)\leq \displaystyle\frac{cd}{c-2}, \\
    (iv)\ \displaystyle\frac{c-2}{c}E(t)\leq\mathcal{E}(t)\leq E(t),
 \end{cases}\nonumber
\end{align}
 where
$c=\min\{p+1,q+1\}>2$.
\end{thm}

In paper \cite{Feng-Guo-Rammaha1}, we also studied the energy decay rates for potential well solutions. 

It is shown in Theorem \ref{thm3-1} the invariance of $\mathcal{W}_1$ under the dynamics. In fact we have the same result for $\mathcal{W}_2$.

\begin{lem} \label{lem3-2-1}
Assume that Assumption \ref{ass} and Assumption \ref{ass1}
hold. Let $1<p\leq 5$ and $q>1$. Assume further
$(u_0,w_0)\in\mathcal{W}_2$ and $\mathcal{E}(0)<d$. Then the weak solution $(u(t),w(t))$ is in $\mathcal{W}_2$ for all $t\in [0,T)$, where $[0,T)$ is the maximal interval of existence.
\end{lem} 

\begin{proof}
Please see the Appendix.
\end{proof}

\begin{rmk}
For initial values coming from the unstable part $\mathcal W_2$ of the potential well, 
we shall state a blow-up result, namely Corollary \ref{cor1}.
\end{rmk}

\vspace{0.1 in}

\subsection{Main Results}
Our first result is the blow-up of solutions if the source terms are stronger than damping terms, and the initial energy is negative.  In order to state our first blow-up result, we need
additional assumptions on the source terms.

\begin{assumption}\label{ass2} 
\hfill
\begin{itemize}
\item
There exists a function $F(u)\in C^1(\mathbb{R})$ such that $F'(u)=f(u)$. In addition, there exist $c_0>0$ and $c_1>3$ such
that
\begin{align}\label{af1}F(u)\geq c_0|u|^{p+1},\ \ uf(u)\geq c_1F(u),\ \ \forall\
u\in\mathbb{R}.
\end{align}
\item There exists a function $H(s)\in C^1(\mathbb{R})$ such that $H'(s)=h(s)$. In addition, there exist $c_2>0$ and $c_3>3$ such that
\begin{align}\label{af2}
H(s)\geq c_2|s|^{q+1},\ \ \ sh(s)\geq c_3H(s), \ \ \forall\
s\in\mathbb{R}.
\end{align}
\end{itemize}
\end{assumption}

The following blow-up result shows that if the initial energy is
negative, and the  source terms are more dominant than their
corresponding damping terms, then every weak solution of
\eqref{PDE} blows up in finite time.
\begin{thm}[\bf Blow-up with negative initial energy]\label{thm6-1}
Suppose that  Assumption \ref{ass} and Assumption \ref{ass2}
hold. Assume $p>m$, $q>r$, and $\mathcal{E}(0)<0$. Then the weak solution $(u(t),w(t))$ of
system \eqref{PDE} blows up in finite time. In particular, 
$$
\limsup_{t \to T^-} (\|\nabla u(t)\|_2^2+|\Delta w(t)|_2^2)= +\infty,
$$
for some $0<T<\infty$.
\end{thm}

\begin{rmk} \label{rmk-p}
Combining the requirements $p>m\geq 1$ and $p\frac{m+1}{m}<6$ from Assumption \ref{ass}, we obtain the restriction that $1<p<5$ and $1\leq m <5$ for the validity of Theorem \ref{thm6-1}.
\end{rmk}

The second result is the blow up of potential well
solutions with positive initial energy. Before stating this result, we shall define several constants. 
Let $y_0>0$ be the unique solution of the equation
\begin{align}\label{F1}
MK_1(p+1)(2y_0)^{\frac{p-1}{2}}+MK_2(q+1)(2y_0)^{\frac{q-1}{2}}=1.
\end{align}
The constants $0<K_1,K_2<\infty$ are given by
\begin{align}\label{3-12}
K_1:=\sup_{u\in
H^1_{\Gamma_0}(\Omega)\backslash\{0\}}\frac{\|u\|^{p+1}_{p+1}}{\|\nabla
u\|^{p+1}_2},\ \ K_2:=\sup_{w\in
H^2_0(\Gamma)\backslash\{0\}}\frac{|w|^{q+1}_{q+1}}{|\Delta
w|^{q+1}_2},
\end{align}
where $K_1$ and $K_2$ are well-defined when $1\leq p\leq 5$ and $q\geq 1$. 
Also, we put
\begin{align} \label{def-dh}
\hat{d}:=y_0-MK_1(2y_0)^{\frac{p+1}{2}}-MK_2(2y_0)^{\frac{q+1}{2}},
\end{align}
where $M>0$ has been introduced in \eqref{3-4}.

\begin{rmk} \label{rmk-d}
We claim that 
\begin{align} \label{dhd}
0<\hat{d}\leq d,
\end{align}
where $d$ is the depth of the potential well, defined in \eqref{3-11}.
The proof of inequality (\ref{dhd}) can be found in the Appendix.
\end{rmk}

Also, we define the positive constant
\begin{align}  \label{defA}
A:= \frac{\lambda}{2(6+\lambda)} y_0,  \;\; \text{where}  \;\; \lambda:=\min\{c_1-3,c_3-3\}>0.
\end{align}

Then we have the following results.

\begin{thm}[\bf Blow-up with positive initial energy] \label{thm6-2}
Suppose that Assumption \ref{ass}, Assumption \ref{ass1} and 
Assumption \ref{ass2} hold.
 Let 
\begin{align}\label{y0}E(0)>y_0,\ \ \mbox{and}\ \
0\leq\mathcal{E}(0)< \min \{A, \hat{d}\}.
\end{align}
Then the weak solution $(u(t),w(t))$ of \eqref{PDE} blows up in
finite time provided $p>m$ and $q>r$. In particular,
$$
\limsup_{t\rightarrow T^-}(\|\nabla u(t)\|_2^2+|\Delta w(t)|_2^2) =+\infty,
$$
for some $0<T<\infty$.
\end{thm}

\begin{cor}\label{cor1}
Suppose that  Assumption \ref{ass}, Assumption \ref{ass1} and 
Assumption \ref{ass2} hold.    Let  $p>m$, $q>r$ and
$$ 
0\leq\mathcal{E}(0)<\min \{A, \hat{d}\}.
$$ 
 If $(u_0,w_0)\in\mathcal{W}_2$, then the weak solution $(u(t),w(t))$ of \eqref{PDE} blows up in
finite time.
\end{cor}

\vspace{0.1 in}

\section{Blow-up of solutions with negative initial energy}\label{blowup1}
In this section, we prove Theorem \ref{thm6-1}, which says that the weak solution of system (\ref{PDE}) blows up in finite time if the source terms are more dominant than damping terms and the initial total energy is negative.
\begin{proof}[Proof of Theorem \ref{thm6-1}]
Let $(u(t),w(t)$ be a weak solution of \eqref{PDE} in the sense of
Definition \ref{def:weaksln}. We define the life span $T$ of such a
solution $(u(t),w(t))$ to be the supremum of all $T^*>0$ such that
$(u(t),w(t))$ is a solution to system \eqref{PDE} in the sense of
Definition \ref{def:weaksln} on $[0, T^*]$. In the following, we will
show that $T$ is finite and obtain an upper bound for the life span
of solutions.

As in \cite{ACCRT,BL3,GR1}, for any $t\in[0,T)$, we define
$$
G(t)=-\mathcal{E}(t),\ \ 
S(t)=\int_\Omega F(u(t))dx+\int_\Gamma H(w(t))d\Gamma,
$$
where the total energy $\mathcal E(t)$ has been introduced in (\ref{3-7}).

Clearly,
\begin{align}
G(t)=-\frac{1}{2}(\|u_t\|^2_2+|w_t|^2_2+\|\nabla u\|^2_2+ \|u\|_2^2 +  |\Delta
w|^2_2)+S(t),\nonumber
\end{align}
which implies
\begin{align}\label{6-1}
\|u_t(t)\|^2_2+|w_t(t)|^2_2+\|\nabla u(t)\|^2_2+\|u(t)\|^2_2+|\Delta w(t)|^2_2= -2G(t)+2S(t).
\end{align}
We define
\begin{align} \label{def-N}
N(t):=\frac{1}{2}\left(\|u(t)\|^2_2+|w(t)|^2_2\right)+\int^t_0\int_\Gamma\gamma u(\tau)\cdot
w(\tau)d\Gamma d\tau,
\end{align}
then we have
\begin{align}\label{6-2}
N'(t)=\int_\Omega u(t)u_t(t)dx+\int_\Gamma
w(t)w_t(t)d\Gamma+\int_\Gamma\gamma u(t)\cdot w(t)d\Gamma.
\end{align}
It follows from  Assumption \ref{ass2} that
\begin{align}\label{6-3}
S(t)\geq c_0\|u(t)\|^{p+1}_{p+1}+c_2|w(t)|^{q+1}_{q+1}.
\end{align}

Since $G(t)=-\mathcal{E}(t)$, the energy identity \eqref{energy-2} can be written as
$$
G(t)=G(0)+\int^t_0\int_\Omega g_1(u_t)u_tdxd\tau+\int^t_0\int_\Gamma
g_2(w_t)w_td\Gamma d\tau.
$$
Then from Assumption \ref{ass} and the regularity of $(u,w)$, we
infer that $G(t)$ is absolutely continuous, and
\begin{align}\label{6-6}
G'(t)&=\int_\Omega g_1(u_t)u_tdx+\int_\Gamma
g_2(w_t)w_td\Gamma \geq\alpha\|u_t(t)\|^{m+1}_{m+1}+\alpha|w_t(t)|^{r+1}_{r+1}\geq 0,
\end{align}
a.e. on $[0,T)$.
Then $G(t)$ is non-decreasing. In view of $G(0)=-\mathcal{E}(0)>0$,
we obtain that for any $0\leq t<T$,
\begin{align}\label{6-7}
0<G(0)\leq G(t)\leq S(t).
\end{align}
Due to \eqref{6-1} and (\ref{6-7}), we obtain
\begin{align}\label{6-7-1}
\|u_t(t)\|^2_2+|w_t(t)|^2_2+\|\nabla u(t)\|^2_2+\|u(t)\|^2_2+|\Delta w(t)|^2_2<2S(t).
\end{align}

We introduce a constant $a$ satisfying
\begin{align}\label{6-4}
0<a<\min\left\{\frac{1}{m+1}-\frac{1}{p+1},
\;\;\frac{1}{r+1}-\frac{1}{q+1},
\;\;\frac{p-1}{2(p+1)},
\;\;\frac{q-1}{2(q+1)}\right\}.
\end{align}

Define
\begin{align} \label{def-Y}
Y(t):=G^{1-a}(t)+\varepsilon N'(t),
\end{align}
where $0<\varepsilon\leq\min\{1,G(0)\}$ will be determined later. 
The function $Y(t)$ is adopted from the important work \cite{GT} by Georgiev and Todorova.

We aim to show that $Y(t)$ approaches infinity in finite time.

First we claim that
\begin{align}\label{6-8}
Y'(t)=(1-a)G^{-a}(t)G'(t)+\varepsilon N''(t),
\end{align}
where
\begin{align}\label{6-9}
N''(t)&=\|u_t(t)\|^2_2+|w_t(t)|^2_2-(\|\nabla
u(t)\|^2_2+|\Delta w(t)|^2_2+\|u(t)\|^2_2)-\int_\Omega
g_1(u_t(t))u(t)dx\nonumber\\
&\quad-\int_\Gamma g_2(w_t(t))w(t)d\Gamma+\int_\Omega
u(t)f(u(t))dx+\int_\Gamma
w(t)h(w(t))d\Gamma\nonumber\\
&\quad+2\int_\Gamma\gamma u(t)\cdot w_t(t)d\Gamma,\ \ \mbox{a.e.}\
\mbox{on}\ [0,T).
\end{align}

We remark that $N''(t)$ can be obtained \emph{formally} by differentiating $N'(t)$ in (\ref{6-2}) and using equations in (\ref{PDE}). But this formal procedure needs to be justified as follows.

By Definition \ref{def:weaksln}, $u_t \in
L^{m+1}(\Omega\times(0,T))$. 
Since $u_0\in H_{\Gamma_0}^1(\Omega) \hookrightarrow L^6(\Omega)$, then $u_0\in L^{m+1}(\Omega)$ for $1\leq m <5$ by referring to Remark \ref{rmk-p}. Then
we have
\begin{align}\label{6-9-1}
\int^T_0\int_\Omega|u|^{m+1}dxdt&= \int^T_0\int_\Omega\left|\int^t_0u_t(\tau)d\tau+u_0\right|^{m+1}dxdt\nonumber\\
&\leq C(T^{m+1}\|u_t(t)\|^{m+1}_{L^{m+1}(\Omega\times(0,T))}+T\|u_0\|^{m+1}_{m+1})<\infty.
\end{align}
This implies $u(t)\in L^{m+1}(\Omega\times(0,T))$ for all $T\geq0$.
We can use the same argument to obtain  $w(t)\in
L^{r+1}(\Gamma\times(0,T))$. Then $u(t)$ and $w(t)$ enjoy the regularity
restrictions imposed on the test functions $\phi(t)$ and $\psi(t)$,
respectively, in Definition \ref{def:weaksln}. Then we can replace $\phi$
by $u$ in \eqref{wkslnwave}, $\psi$ by $w$ in \eqref{wkslnplt} and use
\eqref{6-2} to obtain 
\begin{align}\label{6-10}
N'(t)&=(u,u_t)_\Omega+(w,w_t)_\Gamma+(\gamma
u,w)_\Gamma\nonumber\\
&=\int_\Omega u_0u_1dx+\int_\Gamma(w_0w_1+\gamma
u_0w_0)d\Gamma+\int^t_0(\|u_t\|^2_2+|w_t|^2_2)d\tau\nonumber\\
&\quad-\int^t_0(\|\nabla u\|^2_2+|\Delta
w|^2_2+\|u\|^2_2)d\tau+2\int^t_0\int_\Gamma\gamma u\cdot w_td\Gamma
d\tau-\int^t_0\int_\Omega g_1(u_t)udxd\tau\nonumber\\
&\quad-\int^t_0\int_\Gamma g_2(w_t)wd\Gamma d\tau+\int^t_0\int_\Omega
uf(u)dxd\tau+\int^t_0\int_\Gamma wh(w)d\Gamma d\tau.
\end{align}

In the following, we show that $N'(t)$ is absolutely continuous, and therefore it can be differentiated.

Recall the fact $u\in C([0,t];H^1_{\Gamma_0}(\Omega))$ and the embedding
$H^1_{\Gamma_0}(\Omega)\hookrightarrow L^6(\Omega)$. By Remark \ref{rmk-p}, we know $1<p<5$. Hence, for all $t\in[0,T)$,
\begin{align}\label{6-11}
\int^t_0\left|\int_\Omega
uf(u)dx\right|d\tau&\leq C\int^t_0\int_\Omega(|u|^p+1)|u|dxd\tau<\infty.
\end{align}
Also, since $w\in H^2_0(\Gamma) \hookrightarrow L^{\infty}(\Gamma)$, we have
\begin{align}\label{6-12}
\int^t_0\left|\int_\Gamma wh(w)d\Gamma\right|d\tau\leq
C_T\int^t_0\int_\Gamma|w|^{q+1}d\Gamma d\tau<\infty.
\end{align}

By using the trace theorem, we see that
\begin{align}\label{6-13}
2\int^t_0\left|\int_\Gamma\gamma u\cdot
w_td\Gamma\right|d\tau \leq C\int^t_0 \|\nabla
u\|^2_2 d\tau+\int^t_0 |w_t|^2_2 d\tau<\infty.
\end{align}

Because of (\ref{6-9-1}) 
and the regularity $u_t \in L^{m+1}(\Omega \times (0,T))$, we deduce that for all $t\in[0,T)$,
\begin{align}\label{6-14}
\int^t_0\left|\int_\Omega
g_1(u_t)udx\right|d\tau+\int^t_0\left|\int_\Gamma
g_2(w_t)wd\Gamma\right|d\tau<\infty.
\end{align}
Then \eqref{6-11}-\eqref{6-14} and the regularity of $(u,w)$
imply that all terms on the right-hand side of (\ref{6-10}) are absolutely continuous, 
and thus we can differentiate (\ref{6-10}) to conclude that the claimed formula \eqref{6-9} for $N''(t)$ holds true.

In the following, we aim to find a lower bound for $N''(t)$.

By using Young's inequality, we see that
\begin{align}\label{6-14-1}
2\left|\int_\Gamma\gamma u\cdot
w_t \, d\Gamma \right|  \leq 2|w_t|^2_2+\frac{1}{2}|\gamma u|^2_2.
\end{align}
Now we estimate the term $|\gamma u|^2_2$.
Without loss of generality, we assume the flat portion $\Gamma$ of the boundary is horizontal, and thus the unit normal vector to $\Gamma$ is $\mathbf {n}=(0,0,1)$.
Recall that $\partial\O=\overline{\G_0\cup\G}$ and $u|_{\Gamma_0}=0$. We define a vector field $\mathbf {F}=(0,0,u^2)$ and use the Divergence Theorem to get that
\begin{align}\label{6-14-2}
|\gamma u|^2_2&=\int_\Gamma u^2 d\Gamma=\int_{\Gamma\cup\Gamma_0}u^2d(\Gamma\cup\Gamma_0) =\int_{\Gamma\cup\Gamma_0}\mathbf{F}\cdot\mathbf{n} \,d(\Gamma\cup\Gamma_0) =\int_\Omega \mbox{div}\, \mathbf{F} \, dx\notag\\
&=\int_\Omega (u^2)_zdx=2\int_\Omega uu_zdx\leq \|u\|^2_2+\|u_z\|^2_2\leq \|u\|^2_2+\|\nabla u\|^2_2.
\end{align}
It follows from \eqref{6-14-1} and \eqref{6-14-2} that 
\begin{align}\label{6-14-3}
2 \left|\int_\Gamma\gamma u\cdot
w_td\Gamma \right|  \leq 2|w_t|^2_2+\frac{1}{2}\|u\|^2_2+\frac{1}{2}\|\nabla u\|^2_2.
\end{align}
Then  \eqref{6-9} and \eqref{6-14-3} yield 
\begin{align}\label{e1}
N''(t) &\geq \|u_t\|^2_2-|w_t|^2_2-\frac{3}{2}(\|\nabla u\|^2_2+|\Delta w|^2_2+\|u\|^2_2)-\int_\Omega
g_1(u_t)udx\nonumber\\
&\quad-\int_\Gamma g_2(w_t)wd\Gamma+\int_\Omega
uf(u)dx+\int_\Gamma
wh(w)d\Gamma.
\end{align}
Noting $\|\nabla u\|^2_2+|\Delta
w|^2_2+\|u\|^2_2=-(\|u_t\|^2_2+|w_t|^2_2)+2S(t)-2G(t)$ due to (\ref{6-1}),
and using the assumption $uf(u)\geq
c_1F(u)$, $wh(w)\geq c_3H(w)$ from (\ref{af1})-(\ref{af2}), we infer from \eqref{e1} that
\begin{align}\label{e2}
N''(t)&\geq \frac{5}{2}\|u_t\|^2_2+\frac{1}{2}|w_t|^2_2-3S(t)+3G(t)-\int_\Omega
g_1(u_t)udx-\int_\Gamma g_2(w_t)wd\Gamma\nonumber\\
&\quad+c_1\int_\Omega
F(u)dx+c_3\int_\Gamma
H(w)d\Gamma\nonumber\\
&\geq \frac{5}{2}\|u_t\|^2_2+\frac{1}{2}|w_t|^2_2+3G(t)-\int_\Omega
g_1(u_t)udx-\int_\Gamma g_2(w_t)wd\Gamma + \lambda S(t),
\end{align}
where we let $\lambda:=\min\{c_1-3,c_3-3\}>0$.

By using $g_1(s)s\leq\beta|s|^{m+1}$, H\"{o}lder's inequality and
$p>m$, we have
\begin{align}
\int_\Omega
g_1(u_t)udx&\leq \beta\int_\Omega|u_t|^m|u|dx\leq\beta\|u\|_{m+1}\|u_t\|^m_{m+1} \leq \beta|\Omega|^{\frac{p-m}{(p+1)(m+1)}}\|u\|_{p+1}\|u_t\|^m_{m+1},\nonumber
\end{align}
which along with \eqref{6-3} yields
\begin{align}\label{6-17}
\int_\Omega
g_1(u_t)udx\leq\beta|\Omega|^{\frac{p-m}{(p+1)(m+1)}}c_0^{-\frac{1}{p+1}}S^{\frac{1}{p+1}}(t)\|u_t\|^m_{m+1}=R_1S^{\frac{1}{p+1}}(t)\|u_t\|^m_{m+1},
\end{align}
where the constant $R_1:=\beta |\Omega|^{\frac{p-m}{(p+1)(m+1)}}c_0^{-\frac{1}{p+1}}$.

Then by using Young's
inequality, \eqref{6-6} and \eqref{6-7}, we obtain from \eqref{6-17}
that for any $\delta_1>0$,
\begin{align}\label{6-18}
\int_\Omega
g_1(u_t)udx&\leq R_1S^{\frac{1}{p+1}-\frac{1}{m+1}}(t)S^{\frac{1}{m+1}}(t)\|u_t\|^m_{m+1}\nonumber\\
&\leq G^{\frac{1}{p+1}-\frac{1}{m+1}}(t)\Big[\delta_1S(t)+C_{\delta_1}R_1^{\frac{m+1}{m}}\|u_t\|^{m+1}_{m+1}\Big]\nonumber\\
&\leq \delta_1G^{\frac{1}{p+1}-\frac{1}{m+1}}(t)S(t)+C_{\delta_1}\frac{R_1^{\frac{m+1}{m}}}{\alpha}G'(t)G^{-a}(t)G^{a+\frac{1}{p+1}-\frac{1}{m+1}}(t)\nonumber\\
&\leq \delta_1G^{\frac{1}{p+1}-\frac{1}{m+1}}(0)S(t)+C_{\delta_1}\frac{R_1^{\frac{m+1}{m}}}{\alpha}G'(t)G^{-a}(t)G^{a+\frac{1}{p+1}-\frac{1}{m+1}}(0),
\end{align}
where $a>0$ satisfying (\ref{6-4}), and thus $a+\frac{1}{p+1}-\frac{1}{m+1}<0$. Similarly,
we can obtain for any $\delta_2>0$,
\begin{align}\label{6-19}
\int_\Gamma g_2(w_t)wd\Gamma\leq
\delta_2G^{\frac{1}{q+1}-\frac{1}{r+1}}(0)S(t)+C_{\delta_2}\frac{R_2^{\frac{r+1}{r}}}{\alpha}G'(t)G^{-a}(t)G^{a+\frac{1}{q+1}-\frac{1}{r+1}}(0),
\end{align}
where $R_2:=\beta|\Gamma|^{\frac{q-r}{(q+1)(r+1)}}c_2^{-\frac{1}{q+1}}$.

Inserting \eqref{6-18} and \eqref{6-19} into \eqref{e2}, we obtain
\begin{align}\label{e3}
N''(t)&\geq \frac{5}{2}\|u_t\|^2_2+\frac{1}{2}|w_t|^2_2+3G(t) 
+\left[\lambda - \delta_1G^{\frac{1}{p+1}-\frac{1}{m+1}}(0)-\delta_2G^{\frac{1}{q+1}-\frac{1}{r+1}}(0)\right]S(t)\nonumber\\
&\quad-\left[C_{\delta_1}\frac{R_1^{\frac{m+1}{m}}}{\alpha}G^{a+\frac{1}{p+1}-\frac{1}{m+1}}(0)+
C_{\delta_2}\frac{R_2^{\frac{r+1}{r}}}{\alpha}G^{a+\frac{1}{q+1}-\frac{1}{r+1}}(0)\right]G'(t)G^{-a}(t).
\end{align}

Let us introduce the constants
$\delta_1=\displaystyle\frac{\lambda}{4}G^{\frac{1}{m+1}-\frac{1}{p+1}}(0)$ and 
$\delta_2=\displaystyle\frac{\lambda}{4}G^{\frac{1}{r+1}-\frac{1}{q+1}}(0)$. 
Consequently, we infer from \eqref{6-8} and \eqref{e3} that
\begin{align}\label{6-20}
Y'(t) &\geq \left[(1-a)-\varepsilon
C_{\delta_1}\frac{R_1^{\frac{m+1}{m}}}{\alpha}G^{a+\frac{1}{p+1}-\frac{1}{m+1}}(0)-\varepsilon
C_{\delta_2}\frac{R_2^{\frac{r+1}{r}}}{\alpha}G^{a+\frac{1}{q+1}-\frac{1}{r+1}}(0)\right]G'(t)G^{-a}(t)\nonumber\\
&\;\;\;\;+\frac{5}{2}\varepsilon\|u_t\|^2_2+\frac{1}{2}\varepsilon|w_t|^2_2+3\varepsilon G(t)+ \frac{\lambda}{2}  \varepsilon  S(t).
\end{align}

Noting that $0<a<\frac{1}{2}$, we take $0<\varepsilon<1$
sufficiently small such that
$$
\rho:=(1-a)-\varepsilon
C_{\delta_1}\frac{R_1^{\frac{m+1}{m}}}{\alpha}G^{a+\frac{1}{p+1}-\frac{1}{m+1}}(0)-\varepsilon
C_{\delta_2}\frac{R_2^{\frac{r+1}{r}}}{\alpha}G^{a+\frac{1}{q+1}-\frac{1}{r+1}}(0)\geq0,
$$
 to obtain from \eqref{6-20}  that
\begin{align}\label{6-22}
Y'(t)&\geq \rho
G'(t)G^{-a}(t)+\frac{5}{2}\varepsilon\|u_t\|^2_2+\frac{1}{2}\varepsilon|w_t|^2_2+3\varepsilon G(t)+\frac{\lambda}{2}\varepsilon S(t) >0.
\end{align}
This shows that $Y(t)$ is increasing on $[0,T)$, with
$$
Y(t)=G^{1-a}(t)+\varepsilon N'(t) > Y(0) = G^{1-a}(0)+\varepsilon N'(0).
$$
If $N'(0)\geq0$, then we do not need any  further condition on
$\varepsilon$. But, if $N'(0)<0$, we further take $\varepsilon$ such
that $0<\varepsilon\leq-\frac{G^{1-a}(0)}{2N'(0)}$. In any case, we
have
\begin{align}\label{6-23}
Y(t)\geq \frac{1}{2}G^{1-a}(0)>0,\ \ \mbox{for}\ t\in[0,T).
\end{align}

Finally, we shall prove that the following inequality holds:
\begin{align}\label{6-24}
Y'(t)\geq C\varepsilon^{1+\sigma}Y^\mu(t),\ \ \mbox{for}\ t\in[0,T),
\end{align}
where $C>0$ is a generic constant independent of $\varepsilon$, and
$$
1<\mu=\frac{1}{1-a}<2,\ \ \sigma=\max\{\sigma_1,\sigma_2\}>0,
$$
and
$$
\sigma_1=1-\frac{2}{(1-2a)(p+1)}>0,\ \
\sigma_2=1-\frac{2}{(1-2a)(q+1)}>0,
$$
due to (\ref{6-4}).

Indeed, if $N'(t)\leq 0$ for some $t\in[0,T)$, then for such value
of $t$, we get
\begin{align}\label{6-25}
Y^\mu(t)=[G^{1-a}(t)+\varepsilon N'(t)]^\mu\leq G(t).
\end{align}
Then we infer from \eqref{6-22} and \eqref{6-25} that
$$
Y'(t)\geq 3\varepsilon G(t)\geq 3\varepsilon^{1+\sigma}G(t)\geq
3\varepsilon^{1+\sigma}Y^{\mu}(t).
$$
If $N'(t)> 0$ for some $t\in[0,T)$, we first note that
$Y(t)=G^{1-a}(t)+\varepsilon N'(t)\leq G^{1-a}(t)+N'(t)$, then
\begin{align}\label{6-26}
Y^\mu(t)\leq C \Big[G(t)+[N'(t)]^\mu\Big].
\end{align}
Applying H\"{o}lder's inequality, Young's inequality, trace theorem
and using $1<\mu<2$, we conclude from (\ref{6-2}) that
\begin{align}\label{6-27}
[N'(t)]^\mu&\leq \Big(\|u_t\|_2\|u\|_2+|w_t|_2|w|_2+|\gamma
u|_2|w|_2\Big)^\mu\nonumber\\
&\leq C\Big(\|u_t\|^\mu_2\|u\|^\mu_2+|w_t|^\mu_2|w|^\mu_2+|\gamma
u|^\mu_2|w|^\mu_2\Big)\nonumber\\
&\leq C\Big(\|u_t\|^2_2+\|u\|^{\frac{2\mu}{2-\mu}}_{p+1}+|w_t|^2_2+|w|^{\frac{2\mu}{2-\mu}}_{q+1}+\|\nabla
u\|^2_2+|w|^{\frac{2\mu}{2-\mu}}_{q+1}\Big).
\end{align}
Since $\mu=\frac{1}{1-a}$ and $\sigma_1>0$, it follows that
\begin{align}\label{6-28}
\frac{2\mu}{(2-\mu)(p+1)}-1=\frac{2}{(1-2a)(p+1)}-1=-\sigma_1<0.
\end{align}
Noting $\varepsilon\leq G(0)$, we infer from \eqref{6-3},
\eqref{6-7} and \eqref{6-28} that
\begin{align}\label{6-29}
\|u(t)\|^{\frac{2\mu}{2-\mu}}_{p+1}&= (\|u(t)\|^{p+1}_{p+1})^{\frac{2\mu}{(2-\mu)(p+1)}}\leq
CS(t)^{\frac{2\mu}{(2-\mu)(p+1)}}\nonumber\\
&\leq CS(t)^{\frac{2\mu}{(2-\mu)(p+1)}-1} S(t) \leq
CG^{-\sigma_1}(0)S(t)\leq C\varepsilon^{-\sigma_1}S(t).
\end{align}
In the same way, we have
\begin{align}\label{6-30}
|w(t)|^{\frac{2\mu}{2-\mu}}_{q+1}\leq
C\varepsilon^{-\sigma_2}S(t).
\end{align}
Recall $\sigma=\max\{\sigma_1,\sigma_2\}>0$ and
$\varepsilon^{-\sigma}>1$. By substituting \eqref{6-29} and \eqref{6-30}
into \eqref{6-27}, we get
\begin{align}\label{6-31}
[N'(t)]^\mu&\leq C\Big(\|u_t\|^2_2+|w_t|^2_2+\|\nabla
u\|^2_2+\varepsilon^{-\sigma}S(t)\Big)\nonumber\\
&\leq  C\Big(\|u_t\|^2_2+|w_t|^2_2+S(t)+\varepsilon^{-\sigma}S(t)\Big)\nonumber\\
&\leq C\varepsilon^{-\sigma}\Big(\|u_t\|^2_2+|w_t|^2_2+S(t)\Big),
\end{align}
where \eqref{6-7-1} is used. 
Combining \eqref{6-22}, \eqref{6-26} and \eqref{6-31}, we derive
that
\begin{align}
Y'(t)&\geq  C\varepsilon\Big[G(t)+\|u_t\|^2_2+|w_t|^2_2 +S(t)\Big]\nonumber\\
&\geq C\varepsilon\Big[G(t)+\varepsilon^\sigma[N'(t)]^\mu\Big]\nonumber\\
&\geq C\varepsilon^{1+\sigma}\Big[G(t)+[N'(t)]^\mu\Big]\geq
C\varepsilon^{1+\sigma}Y^\mu(t),\nonumber
\end{align}
for all values of $t\in[0,T)$ for which $N'(t)>0$. Then in any case,
\eqref{6-24} holds true.

It follows  from \eqref{6-23} and \eqref{6-24} that the maximum life span $T$ is
necessarily finite with
\begin{align}\label{6-32}
T<C\varepsilon^{-(1+\sigma)}Y^{-\frac{a}{1-a}}(0)\leq
C\varepsilon^{-(1+\sigma)}G^{-a}(0).
\end{align}
Notice that, at the blow-up time $T$, the quadratic energy must approache infinity:
\begin{align}\label{6-33}
\limsup_{t\rightarrow T^-} E(t) = +\infty.
\end{align}
We claim 
\begin{align} \label{6-34}
\limsup_{t\rightarrow T^-} (\|\nabla u(t)\|_2^2 + |\Delta w(t)|_2^2) = +\infty.
\end{align}
In fact, (\ref{6-1}) shows that
\begin{align}  \label{6-35}
E(t) = - G(t) + S(t)  < S(t),
\end{align}
because $G(t)>0$ on $[0,T)$. It follows from (\ref{6-34})-(\ref{6-35}) that
\begin{align} \label{6-36}
\limsup_{t\rightarrow T^-} S(t) = + \infty.
\end{align}
Recall $p<5$ from Remark \ref{rmk-p}, then we have 
\begin{align*}
\|\nabla u(t)\|_2^2 + |\Delta w(t)|_2^2
\geq C (\|u\|_{p+1}^{p+1} + |w|_{q+1}^{q+1}) \geq C S(t),
\end{align*}
and along with (\ref{6-36}), we obtain (\ref{6-34}). The proof is completed.
\end{proof}

\vspace{0.1 in}

\section{Blow-up of solutions with positive initial energy}\label{blowup2}
This section is devoted to proving Theorem \ref{thm6-2} and its corollary. 
These results state that the weak solution of system (\ref{PDE}) blows up in finite time if the source terms dominate the damping terms, and the initial total energy $\mathcal E(0)$ is positive but sufficiently small, 
and the initial quadratic energy $E(0)$ is sufficiently large. The basic idea comes from the potential well theory.

\subsection{Proof of Theorem \ref{thm6-2}}
\begin{proof}[Proof of Theorem \ref{thm6-2}]
We use some ideas from \cite{GT,GRS2,V4}. We define the life span $T$ of such a
solution $(u(t),w(t))$ to be the supremum of all $T^*>0$ such that
$(u(t),w(t))$ is a solution to system \eqref{PDE} in the sense of
Definition \ref{def:weaksln} on $[0, T^*]$.

By using \eqref{3-4} and \eqref{3-12}, we have that for $t\in [0,T),$
\begin{align}\label{b8}
\mathcal{E}(t)&=E(t)-\int_\Omega F(u)dx-\int_\Gamma H(w)d\Gamma\nonumber\\
&\geq E(t)-M\|u(t)\|^{p+1}_{p+1}-M|w(t)|^{q+1}_{q+1}\nonumber\\
&\geq E(t)-MK_1\|\nabla u\|^{p+1}_2-MK_2|\Delta w|^{q+1}_2\nonumber\\
&\geq E(t)-MK_1(2E(t))^{\frac{p+1}{2}}-MK_2(2E(t))^{\frac{q+1}{2}},\ \ \mbox{for}\ \mbox{all}\ t\in[0,T).
\end{align}

We define the function $F_1:\mathbb{R}^+\to\mathbb{R}$ by
\begin{align}\label{b1}
F_1(y):=y-MK_1(2y)^{\frac{p+1}{2}}-MK_2(2y)^{\frac{q+1}{2}},
\end{align}
where the positive constants $K_1,K_2$ were given in  \eqref{3-12} and $M>0$ was introduced in \eqref{3-4}.
Then \eqref{b8} is equivalent to the following form:
\begin{align}\label{b9}
\mathcal{E}(t)\geq F_1(E(t)),\ \ \forall\ t\in [0,T).
\end{align}
In view of $p,q>1$, we see that $F_1(y)$ is continuously differentiable, concave and
has its maximum at $y=y_0>0$, where $y_0$ satisfies
\begin{align}\label{b2}
MK_1(p+1)(2y_0)^{\frac{p-1}{2}}+MK_2(q+1)(2y_0)^{\frac{q-1}{2}}=1.
\end{align}
We define
\begin{align}\label{b3}
\hat{d}:=\sup_{[0,\infty)}F_1(y)=F_1(y_0)=y_0-MK_1(2y_0)^{\frac{p+1}{2}}-MK_2(2y_0)^{\frac{q+1}{2}}. 
\end{align}

Since the function $F_1(y)$ has its maximum value at $y=y_0$, then $F_1(y)$ is decreasing if $y>y_0$. As $0\leq\mathcal{E}(0)<\hat{d}=F_1(y_0)$, then there exists a unique constant $y_1$ such that
\begin{align}\label{b10}
F_1(y_1)=\mathcal{E}(0),\ \ \mbox{with}\ y_1>y_0>0.
\end{align}
Then it follows from \eqref{b9} that
\begin{align}\label{b11}
\hat{d}=F_1(y_0)>F_1(y_1)=\mathcal{E}(0)\geq \mathcal{E}(t)\geq F_1(E(t)),\ \ \forall\ t\in [0,T).
\end{align}
Note that $F_1(y)$ is continuous and decreasing if $y>y_0$, and $E(t)$ is also continuous. Since we assume $E(0)>y_0$, we infer from \eqref{b11} that
\begin{align}\label{b12}
E(t)\geq y_1>y_0,\ \ \forall\ t\in[0,T).
\end{align} 
As in Section \ref{blowup1}, we define 
$$
N(t):=\frac{1}{2}\left(\|u(t)\|^2_2+|w(t)|^2_2\right)+\int^t_0\int_\Gamma\gamma u(\tau)\cdot
w(\tau)d\Gamma d\tau,
$$
and
\begin{align} \label{SS}
S(t):=\int_\Omega F(u(t))dx+\int_\Gamma H(w(t))d\Gamma.
\end{align}
Let us define 
\begin{align}    \label{defGG}
\mathcal G(t) := A-\mathcal{E}(t),
\end{align}
where the constant $A$ has been introduced in (\ref{defA}).

Due to the energy identity (\ref{energy-2}), we know $\mathcal E'(t)\leq 0$,
and thus $\mathcal G'(t)\geq 0$, i.e., $\mathcal G(t)$ is non-decreasing in time. 
Since we assume $\mathcal E(0)<A$, then $\mathcal G(0) := A-\mathcal{E}(0)>0$. Therefore, we have
\begin{align} \label{GG0}
\mathcal G(t) \geq  \mathcal G(0) >0,  \;\;\text{for all} \;\; t\in [0,T).
\end{align}

We consider the function
\begin{align}\label{b14}
Y(t):=\mathcal G^{1-a}(t)+\varepsilon N'(t),
\end{align}
for some $a\in (0,\frac{1}{2})$ satisfying (\ref{6-4})
and $\varepsilon>0$. We plan to show that $Y(t)$ approaches infinity in finite time, by choosing $\varepsilon$ sufficiently small. 
Adopting the same arguments as \eqref{6-8}, we see that
\begin{align}\label{b15}
Y'(t)=(1-a)\mathcal G^{-a}(t)\mathcal G'(t)+\varepsilon N''(t),
\end{align}
where
\begin{align}\label{b16}
N''(t)&=\|u_t(t)\|^2_2+|w_t(t)|^2_2-(\|\nabla
u(t)\|^2_2+\|u(t)\|_2^2 + |\Delta w(t)|^2_2)-\int_\Omega
g_1(u_t(t))u(t)dx\nonumber\\
&\quad-\int_\Gamma g_2(w_t(t))w(t)d\Gamma+\int_\Omega
u(t)f(u(t))dx+\int_\Gamma
w(t)h(w(t))d\Gamma\nonumber\\
&\quad+2\int_\Gamma\gamma u(t)\cdot w_t(t)d\Gamma,\ \ \mbox{a.e.}\
\mbox{on}\ [0,T).
\end{align}
Following the same estimates as in \eqref{6-9-1}-\eqref{e1}, we obtain 
\begin{align}\label{b16-1}
N''(t) &\geq \|u_t\|^2_2-|w_t|^2_2-\frac{3}{2}(\|\nabla u\|^2_2+|\Delta w|^2_2+\|u\|^2_2)-\int_\Omega
g_1(u_t)udx\nonumber\\
&\quad-\int_\Gamma g_2(w_t)wd\Gamma+\int_\Omega
uf(u)dx+\int_\Gamma
wh(w)d\Gamma.
\end{align}
Since 
$$
\|\nabla u\|^2_2+|\Delta w|^2_2+\|u\|^2_2=2A-\|u_t\|^2_2-|w_t|^2_2+2S(t)-2\mathcal G(t),
$$
and noting $uf(u)\geq c_1F(u)$, $wh(w)\geq c_3H(w)$,
then we conclude from \eqref{b16-1} that
\begin{align*}
N''(t)&\geq\frac{5}{2}\|u_t\|^2_2+\frac{1}{2}|w_t|^2_2-3A+3\mathcal G(t)+\lambda S(t)-\int_\Omega g_1(u_t)udx-\int_\Gamma g_2(w_t)wd\Gamma,
\end{align*}
where $\lambda:=\min\{c_1-3,c_3-3\}>0$.  Because of (\ref{SS}), (\ref{defGG}) and (\ref{3-7}), we have
$$
S(t)=\mathcal G(t)-A+E(t).
$$
Then
\begin{align*}
N''(t)&\geq\frac{5}{2}\|u_t\|^2_2+\frac{1}{2}|w_t|^2_2-\left(3+\frac{\lambda}{2}\right)A+\left(3+\frac{\lambda}{2}\right)\mathcal G(t)+\frac{\lambda}{2}E(t)+\frac{\lambda}{2} S(t)\nonumber\\
&\quad-\int_\Omega g_1(u_t)udx-\int_\Gamma g_2(w_t)wd\Gamma,\ \ \mbox{for}\ t\in[0,T).
\end{align*}
By recalling (\ref{b12}) and (\ref{defA}), one has
$$
\frac{\lambda}{4}E(t) >  \frac{\lambda}{4}y_0 = \left(3+\frac{\lambda}{2}\right)A,
\;\;\text{for all}\;\; t\in [0,T).
$$
It follows that
\begin{align}\label{b26}
N''(t)&\geq\frac{5}{2}\|u_t\|^2_2+\frac{1}{2}|w_t|^2_2+\left(3+\frac{\lambda}{2}\right)\mathcal G(t)+\frac{\lambda}{4}E(t)+\frac{\lambda}{2} S(t)\nonumber\\
&\quad-\int_\Omega g_1(u_t)udx-\int_\Gamma g_2(w_t)wd\Gamma,\ \ \mbox{for}\ t\in[0,T).
\end{align}

Recalling $\lambda=\min\{c_1-3,c_3-3\}>0$, we have $A=\frac{\lambda}{12+2\lambda}y_0<y_0$.
Then it is concluded from \eqref{b12} that  
\begin{align} \label{b27}
\mathcal G(t)&=A-\mathcal{E}(t)=A-E(t)+S(t) < y_0-y_1+S(t) <S(t),
\end{align}
for $t\in [0,T)$. Moreover, we infer from (\ref{defGG}) and the energy inequality (\ref{energy-2}) that
\begin{align} \label{b27'}
\mathcal G'(t) = - \mathcal E'(t) 
= \int_{\Omega} g_1(u_t) u_t dx + \int_{\Gamma} g_2(w_t) w_t d\Gamma 
\geq \alpha \|u_t\|_{m+1}^{m+1} + \alpha |w_t|_{r+1}^{r+1} \geq 0,
\end{align}
for all $t\in [0,T)$. Then, we use the same arguments as in \eqref{6-18} and \eqref{6-19} to obtain from \eqref{b27}-(\ref{b27'}) that 
\begin{align}
&\int_\Omega g_1(u_t)udx\leq \delta_1 \mathcal G^{\frac{1}{p+1}-\frac{1}{m+1}}(0)S(t)+C_{\delta_1}\frac{R_1^{\frac{m+1}{m}}}{\alpha}\mathcal G'(t) \mathcal G^{-a}(t)\mathcal G^{a+\frac{1}{p+1}-\frac{1}{m+1}}(0),   \label{b28} \\
&\int_\Gamma g_2(w_t)wd\Gamma\leq
\delta_2 \mathcal G^{\frac{1}{q+1}-\frac{1}{r+1}}(0)S(t)+C_{\delta_2}\frac{R_2^{\frac{r+1}{r}}}{\alpha}\mathcal G'(t) \mathcal G^{-a}(t) \mathcal G^{a+\frac{1}{q+1}-\frac{1}{r+1}}(0),   \label{b29}
\end{align}
for any $\delta_1, \delta_2>0$, where the constant $a$ satisfies (\ref{6-4}).

Substituting \eqref{b26}, \eqref{b28} and \eqref{b29} into \eqref{b15}, we obtain
\begin{align}\label{b30}
Y'(t)&=(1-a)\mathcal G^{-a}(t) \mathcal G'(t)+\varepsilon N''(t)\nonumber\\
&\geq \left[(1-a)-\varepsilon C_{\delta_1}\frac{R_1^{\frac{m+1}{m}}}{\alpha} \mathcal G^{a+\frac{1}{p+1}-\frac{1}{m+1}}(0)-\varepsilon C_{\delta_2}\frac{R_2^{\frac{r+1}{r}}}{\alpha} \mathcal G^{a+\frac{1}{q+1}-\frac{1}{r+1}}(0)\right]\mathcal G^{-a}(t) \mathcal G'(t)\nonumber\\
&\quad+\varepsilon\left[\frac{5}{2}\|u_t\|^2_2+\frac{1}{2}|w_t|^2_2+\left(3+\frac{\lambda}{2}\right)\mathcal G(t)+\frac{\lambda}{4}E(t)\right]\nonumber\\
&\quad+\varepsilon\left[\frac{\lambda}{2}-\delta_1 \mathcal G^{\frac{1}{p+1}-\frac{1}{m+1}}(0)-\delta_2 \mathcal G^{\frac{1}{q+1}-\frac{1}{r+1}}(0)
\right]S(t).
\end{align}
At this point, we select $\delta_1,\delta_2>0$ such that
$$
\frac{\lambda}{2}-\delta_1 \mathcal G^{\frac{1}{p+1}-\frac{1}{m+1}}(0)-\delta_2 \mathcal G^{\frac{1}{q+1}-\frac{1}{r+1}}(0)\geq\frac{\lambda}{4}.
$$
For these fixed values of $\delta_1,\delta_2>0$, we choose $\varepsilon>0$ sufficiently small that
$$
(1-a)-\varepsilon C_{\delta_1}\frac{R_1^{\frac{m+1}{m}}}{\alpha} \mathcal G^{a+\frac{1}{p+1}-\frac{1}{m+1}}(0)-\varepsilon C_{\delta_2}\frac{R_2^{\frac{r+1}{r}}}{\alpha} \mathcal G^{a+\frac{1}{q+1}-\frac{1}{r+1}}(0)\geq \frac{1}{2}(1-a).
$$
Then, from \eqref{b30} we obtain
\begin{align}\label{b31}
Y'(t)&\geq \varepsilon \left[\frac{5}{2}\|u_t\|^2_2+\frac{1}{2}|w_t|^2_2+\left(3+\frac{\lambda}{2}\right)\mathcal G(t)+\frac{\lambda}{4}E(t)\right]+\frac{\lambda}{4}\varepsilon S(t) >0,
\end{align}
for all $t\in [0,T)$. Therefore, $Y(t)$ is increasing on $[0,T)$, with
$$
Y(t)=\mathcal G^{1-a}(t)+\varepsilon N'(t) > Y(0) = \mathcal G^{1-a}(0)+\varepsilon N'(0).
$$
Similar to \eqref{6-23}, one can choose $\varepsilon$ sufficiently small such that 
\begin{align}\label{b33}
Y(t) \geq \frac{1}{2} \mathcal G^{1-a}(0)>0,\ \ \mbox{for}\ t\in[0,T).
\end{align}
Now, we claim
\begin{align}\label{nODE}
Y'(t)\geq C\varepsilon^{1+\sigma}Y^\mu(t),\ \ \mbox{for}\ t\in[0,T),
\end{align}
where $\mu:=\frac{1}{1-a}\in (1,2)$ and 
$\sigma:=\max\{\sigma_1,\sigma_2\}>0$ with $\sigma_1=1-\frac{2}{(1-2a)(p+1)}>0$ and $\sigma_2=1-\frac{2}{(1-2a)(q+1)}>0$. By solving differential inequality (\ref{nODE}) with (\ref{b33}), we deduce that 
the maximum life span $T$ is necessarily finite with
\begin{align*}
T<C\varepsilon^{-(1+\sigma)}Y^{-\frac{a}{1-a}}(0)\leq
C\varepsilon^{-(1+\sigma)}\mathcal G^{-a}(0).
\end{align*}

To prove (\ref{nODE}), we use the following argument. If $N'(t)\leq 0$ for some $t\in[0,T)$, then for such value
of $t$, we get
\begin{align}  \label{b34}
Y^\mu(t)=[\mathcal G^{1-a}(t)+\varepsilon N'(t)]^\mu\leq \mathcal G(t).
\end{align}
Then we infer from \eqref{b31} and \eqref{b34} that
$$
Y'(t)\geq 3\varepsilon \mathcal G(t)\geq 3\varepsilon^{1+\sigma}\mathcal G(t)\geq
3\varepsilon^{1+\sigma}Y^{\mu}(t),
$$
for any value of $t$ such that $N'(t)\leq 0$. 

If $N'(t)> 0$ for some $t\in[0,T)$, then
\begin{align}\label{b35}
Y^\mu(t)\leq C \Big[\mathcal G(t)+[N'(t)]^\mu\Big].
\end{align}

We know that $S(t)> \mathcal G(t)\geq \mathcal G(0)>0$ by (\ref{b27}) and (\ref{GG0}). Let $\varepsilon \leq \mathcal G(0)$.
Then, following estimates \eqref{6-27}-\eqref{6-30}, we can derive 
\begin{align}  \label{b36}
[N'(t)]^\mu \leq C(\|u_t\|_2^2 + |w_t|_2^2  + \|\nabla u\|_2^2 + \varepsilon^{-\sigma} S(t))
\leq  C \varepsilon^{-\sigma} (E(t)+S(t)).
\end{align}
Combining \eqref{b31}, \eqref{b36} and (\ref{b35}), we arrive at
\begin{align*}
Y'(t)&\geq C\varepsilon\Big[\|u_t\|^2_2+|w_t|^2_2+\mathcal G(t)+E(t)+S(t)\Big]\\
&\geq C\varepsilon\Big[\mathcal G(t)+\varepsilon^\sigma  [N'(t)]^\mu\Big] \geq C\varepsilon^{1+\sigma}\Big[\mathcal G(t)+ [N'(t)]^\mu\Big]\geq  C\varepsilon^{1+\sigma}Y^\mu(t),
\end{align*}
for any value of $t$ such that $N'(t)> 0$. As a result, we conclude that (\ref{nODE}) holds for all values of $t\in [0,T)$.

Finally, by using the same argument as in Section \ref{blowup1}, we conclude 
$\limsup_{t\to T^{-}} (\|\nabla u(t)\|_2^2+|\Delta w(t)|_2^2)=+\infty$.
This completes the proof. 
\end{proof}

\vspace{0.1 in}

\subsection{Proof of Corollary \ref{cor1}}
Corollary \ref{cor1} states that the weak solution of system (\ref{PDE}) blows up in finite time if the source terms exceed the damping terms, and the initial total energy $\mathcal E(0)$ is positive but sufficiently small, 
and the initial data are from $\mathcal W_2$, i.e., the unstable part of the potential well.

\begin{proof}[Proof of Corollary \ref{cor1}]
It suffices to show that if $(u_0,w_0)\in\mathcal{W}_2$, then $E(0)>y_0$.

Since $(u_0,w_0)\in\mathcal{W}_2$, then by the definition of $\mathcal{W}_2$, we get
$$
\|\nabla u_0\|^2_2+|\Delta w_0|^2_2+\|u_0\|^2_2<(p+1)\int_\Omega F(u_0)dx+(q+1)\int_\Gamma H(w_0)d\Gamma,
$$
which together with \eqref{3-4} implies
\begin{align}\label{z0}
\|\nabla u_0\|^2_2+|\Delta w_0|^2_2+\|u_0\|^2_2<M(p+1)\|u_0\|^{p+1}_{p+1}+M(q+1)|w_0|^{q+1}_{q+1}.
\end{align}
Let $X:=H^1_{\Gamma_0}(\Omega)\times H^2_0(\Gamma)$ and recall the definition of $K_1,K_2$ in \eqref{3-12}. Then we obtain from \eqref{z0} that
\begin{align}\label{z1}
\|(u_0,w_0)\|^2_X&<M(p+1)K_1\|\nabla u_0\|^{p+1}_{2}+ M(q+1)K_2|w_0|^{q+1}_2\nonumber\\
&\leq M(p+1)K_1\|(u_0,w_0)\|^{p+1}_{X}+ M(q+1)K_2\|(u_0,w_0)\|^{q+1}_X.
\end{align}
We divide both sides of \eqref{z1} by $\|(u_0,w_0)\|^2_X$ to reach
$$
 M(p+1)K_1\|(u_0,w_0)\|^{p-1}_{X}+ M(q+1)K_2\|(u_0,w_0)\|^{q-1}_X>1.
$$
This along with \eqref{F1} gives 
\begin{align*}
& MK_1(p+1)\left(\|(u_0,w_0)\|^{2}_{X}\right)^{\frac{p-1}{2}}+ MK_2(q+1)\left(\|(u_0,w_0)\|^{2}_X\right)^{\frac{q-1}{2}}\\
&\quad>1=MK_1(p+1)(2y_0)^{\frac{p-1}{2}}+MK_2(q+1)(2y_0)^{\frac{q-1}{2}}.
\end{align*}
Since $p,q>1$, then we have
\begin{align}\label{z2}
\|(u_0,w_0)\|^{2}_{X}>2y_0,
\end{align}
which implies that
$E(0)>y_0$. Then, using Theorem \ref{thm6-2}, we obtain the blow-up of weak solutions in finite time. 
\end{proof}

\vspace{0.1 in}

\section{Appendix}

\subsection{Proof of Lemma \ref{lem3-2-1}}

\begin{proof} [Proof of Lemma \ref{lem3-2-1}]
Because of energy equality (\ref{energy-2}) and (\ref{3-9}),
we have
\begin{align}\label{w2-2}
\mathcal{J}(u,w)\leq\mathcal{E}(t)\leq\mathcal{E}(0)<d,\ \ \mbox{for}\  t\in[0,T).
\end{align}
Then  $(u(t),w(t))\in\mathcal{W}$. 
To prove $(u(t),w(t))\in\mathcal{W}_2$, we argue by contradiction.
We assume that there exists $t_1\in(0,T)$ such that
$(u(t_1),w(t_1))\notin\mathcal{W}_2$. Recalling
$\mathcal{W}_1\cup\mathcal{W}_2=\mathcal{W}$ and
$\mathcal{W}_1\cap\mathcal{W}_2=\emptyset$, then we obtain that
$(u(t_1),w(t_1))\in\mathcal{W}_1$.

By \eqref{3-6} and the mean value theorem, we can get that for any
$t_0\in [0,T)$,
\begin{align}
\int_\Omega|F(u(t))-F(u(t_0))|dx&\leq C\int_\Omega
(|u(t)|^p+|u(t_0)|^p)|u(t)-u(t_0)|dx\nonumber\\
&\leq C(\|u(t)\|^p_{p+1}+\|u(t_0)\|^p_{p+1})\|u(t)-u(t_0)\|_{p+1}.\nonumber
\end{align}
Noting $p\leq 5$, and using the embedding
$H^1_{\Gamma_0}(\Omega)\hookrightarrow L^6(\Omega)$ and the
regularity of the weak solution $u\in
C([0,T);H^1_{\Gamma_0}(\Omega))$, we conclude that
$\int_\Omega F(u(t)) dx\to \int_\Omega F(u(t_0)) dx$ as $t\to
t_0$, which implies that the function $t\mapsto\int_\Omega F(u(t))dx$ is
continuous on $[0,T)$.
Similarly, the continuity of  the function $t\mapsto\int_\Gamma
H(w(t))d\Gamma$ is obtained on $[0,T)$.

Since  $(u(0),w(0))\in\mathcal{W}_2$ and
$(u(t_1),w(t_1))\in\mathcal{W}_1$, then by the continuity and the intermediate value theorem, we know that there exists
$s\in (0,t_1]$ such that
\begin{align}\label{w2-3}
\|\nabla u(s)\|^2_2+|\Delta w(s)|^2_2+\|u(s)\|^2_2=(p+1)\int_\Omega
F(u(s))dx+(q+1)\int_\Gamma H(w(s))d\Gamma.
\end{align}
Define $t^*$ be the infinimum point over $s\in(0,t_1]$ satisfying
\eqref{w2-3}. Then $t^*\in(0,t_1]$ and $(u(t),w(t))\in\mathcal{W}_2$
for
any $t\in [0,t^*)$. Two cases are considered as follows:\\
\emph{Case 1.}  $(u(t^*),w(t^*))\neq(0,0)$. Since \eqref{w2-3} holds
for $t^*$, then $(u(t^*),w(t^*))\in\mathcal{N}$. We can get from
\eqref{3-11} that $\mathcal{J}(u(t^*),w(t^*))\geq d$. Since
$\mathcal{E}(t)\geq\mathcal{J}(u(t),w(t))$ for any $t\in[0,T)$,
$\mathcal{E}(t^*)\geq d$ is obtained. This contradicts \eqref{w2-2}.\\
\emph{Case 2.} $(u(t^*),w(t^*))=(0,0)$. Note that
$(u(t),w(t))\in\mathcal{W}_2$ for any $t\in[0,t^*)$. We conclude
from \eqref{3-4} that for any $t\in[0,t^*)$,
\begin{align}
\|\nabla u(t)\|^2_2+|\Delta w(t)|^2_2+\|u(t)\|^2_2\leq
C(\|u(t)\|^{p+1}_{p+1}+|w(t)|^{q+1}_{q+1})\leq C(\|\nabla
u(t)\|^{p+1}_2+|\Delta w(t)|^{q+1}_2),\nonumber
\end{align}
which implies
$$
\|(u(t),w(t))\|^2_X<C(\|(u(t),w(t))\|^{p+1}_X+\|(u(t),w(t))\|^{q+1}_X),\
\ t\in [0,t^*),
$$
where $X=H^1_{\Gamma_0}(\Omega)\times H^2_0(\Gamma)$. Then, for any $t\in [0,t^*)$, we see that
$$
\|(u(t),w(t))\|^{p-1}_X+\|(u(t),w(t))\|^{q-1}_X>\frac{1}{C}.
$$
This gives us $\|(u(t),w(t))\|_X>s_0$, for any $t\in [0,t^*)$,
where $s_0>0$ is the unique positive solution of
$s^{p-1}+s^{q-1}=\frac{1}{C}$, where $p,q>1$. It follows from the
continuity of the weak solution $(u(t),w(t))$ that
$\|(u(t^*),w(t^*))\|_X\geq s_0>0$. This contradicts that $(u(t^*),w(t^*))=(0,0)$. Therefore, $(u(t),w(t))\in \mathcal{W}_2$ for all $t\in [0,T)$.
\end{proof}

\subsection{Proof of inequality (\ref{dhd})}
\begin{proof} [Proof of (\ref{dhd})]

We justify that $0<\hat d\leq d$. Indeed, by using (\ref{def-dh}) and (\ref{F1}), we have
\begin{align*}
\hat{d}&=y_0-MK_1(2y_0)^{\frac{p+1}{2}}-MK_2(2y_0)^{\frac{q+1}{2}}\\
&=y_0-\frac{2y_0}{p+1}\cdot MK_1(p+1)(2y_0)^{\frac{p-1}{2}}-
\frac{2y_0}{q+1}\cdot MK_2(q+1)(2y_0)^{\frac{q-1}{2}}\\
&\geq y_0-\max\left\{\frac{2y_0}{p+1},\frac{2y_0}{q+1}\right\}\left[MK_1(p+1)(2y_0)^{\frac{p-1}{2}}+MK_2(q+1)(2y_0)^{\frac{q-1}{2}}\right]\\
&=y_0-\max\left\{\frac{2y_0}{p+1},\frac{2y_0}{q+1}\right\}=y_0\cdot\min\left\{\frac{p-1}{p+1},\frac{q-1}{q+1}\right\},
\end{align*}
which, using the fact $p,q>1$, implies $\hat{d}>0$. 

Let $X=H^1_{\Gamma_0}(\Omega)\times H^2_0(\Gamma)$. It follows  from \eqref{3-4}, \eqref{3-8} and \eqref{3-12} that
\begin{align}\label{b5}
\mathcal{J}(u,w)&\geq\frac{1}{2}(\|\nabla u\|^2_2+\|u\|_2^2 +  |\Delta
w|^2_2)-M(\|u\|^{p+1}_{p+1}+|w|^{q+1}_{q+1})\nonumber\\
&\geq \frac{1}{2}(\|\nabla u\|^2_2 + \|u\|_2^2    +|\Delta w|^2_2)-MK_1\|\nabla
u\|^{p+1}_2-MK_2|\Delta w|^{q+1}_2\nonumber\\
&\geq \frac{1}{2}\|(u,w)\|^2_X-MK_1\|(u,w)\|^{p+1}_X-MK_2\|(u,w)\|^{q+1}_X\nonumber\\
&:=\Lambda(\|(u,w)\|_X),
\end{align}
with 
$$
\Lambda(y)=\frac{1}{2}y^2-MK_1y^{p+1}-MK_2y^{q+1}.
$$
Since $p,q>1$, then
$$
\Lambda'(y)=y[1-MK_1(p+1)y^{p-1}-MK_2(q+1)y^{q-1}],
$$
has only one positive zero at $y^*$, where $y^*$ satisfies
\begin{align}\label{b5-1}
 MK_1(p+1)(y^*)^{p-1}+MK_2(q+1)(y^*)^{q-1}=1.
\end{align}
It is easy to verify that $\Lambda(y)$ has maximum value at
$y=y^*$, i.e.,
\begin{align*}
\Lambda(y^*)=\sup_{[0,\infty)}\Lambda(y)=\frac{1}{2}(y^*)^2-MK_1(y^*)^{p+1}-MK_2(y^*)^{q+1}.
\end{align*}
It follows from \eqref{F1} and \eqref{b5-1} that 
$(y^*)^2=2y_0$. 
Therefore,
\begin{align} \label{b6}
\Lambda(y^*)&= y_0-MK_1(2y_0)^{\frac{p+1}{2}}-MK_2(2y_0)^{\frac{q+1}{2}}=\hat{d}.
\end{align}

From \eqref{b5}, we obtain
$$
\mathcal{J}(\lambda (u,w))\geq \Lambda(\lambda\|(u,w)\|_X),\ \mbox{for}\ \mbox{all}\ \lambda\geq0.
$$
It follows that
$$
\sup_{\lambda\geq0}\mathcal{J}(\lambda (u,w))\geq \Lambda(y^*).
$$
Then we infer from \eqref{3-11} and \eqref{b6} that
$$
d=\inf_{(u,w)\in X\backslash(0,0)}\sup_{\lambda\geq0}\mathcal{J}(\lambda (u,w))\geq \Lambda(y^*)=\hat{d}.
$$
This shows that $\hat{d}$ is not larger than the depth $d$ of the potential well.
\end{proof}

\bibliographystyle{abbrv}
\def\cprime{$'$}

\end{document}